\documentclass{amsart}

\usepackage{diagrams}
\usepackage{amssymb}
\usepackage{graphicx}
\usepackage{stmaryrd}
\usepackage{tikz}

\newtheorem{thm}{Theorem}[section]
\newtheorem{cor}{Corollary}[section]

\newtheorem{prop}{Proposition}[section]

\theoremstyle{definition}
\newtheorem{defn}{Definition}[section]
\newtheorem{ex}{Example}[section]

\title{The Topological Complexity of Finite Models of Spheres}
\author{Shelley Kandola}

\begin{document}

\maketitle

\begin{abstract}
In \cite{Farber2001}, Farber defined topological complexity (TC) to be the minimal number of continuous motion planning rules required to navigate between any two points in a topological space.
Papers by \cite{Gonzalez2017} and \cite{Fernandez-Ternero2018} define notions of topological complexity for simplicial complexes.
In \cite{Tanaka2018}, Tanaka defines a notion of topological complexity, called combinatorial complexity, for finite topological spaces.
As is common with papers discussing topological complexity, each includes a computation of the TC of some sort of circle.
In this paper, we compare the TC of models of $S^1$ across each definition, exhibiting some of the nuances of TC that become apparent in the finite setting.
In particular, we show that the TC of finite models of $S^1$ can be 3 or 4 and that the TC of the minimal finite model of any $n$-sphere is equal to 4.
Furthermore, we exhibit spaces weakly homotopy equivalent to a wedge of circles with arbitrarily high TC.
\end{abstract}

\section{Introduction}
Farber introduced the notion of topological complexity in \cite{Farber2001} as it relates to motion planning in robotics.
Informally, the topological complexity of a robot's space of configurations represents the minimal number of continuous motion planning rules required to instruct that robot to move from one position into another position.
Although topological complexity was originally defined for robots with a smooth, infinite range of motion (e.g. products of spheres or real projective space), it makes sense to consider the topological complexity of finite topological spaces.
For example, one could determine the topological complexity of a finite state machine or a robot powered by stepper motors.

This paper was motivated by learning that the topological complexity of $S^1$ does not agree with that of its minimal finite model.
It is well-known from \cite{Farber2001} that $\text{TC}(S^n) = 2$ for $n$ odd and 3 for $n \geq 2$ even.
Upon further inspection, it became clear the not all models of $S^1$ have the same topological complexity.
In \cite{Tanaka2018}, Tanaka proves that $\text{TC}(\mathbb{S}^1) = 4$, where $\mathbb{S}^1$ is the minimal model of $S^1$ comprising four points. 
This value drops as the size of the model of $S^1$ increases.

\begin{thm}
\label{TCS13}
For the finite model $\mathbb{S}^1_n$ of $S^1$ comprising $2n$ points for $n >2$, \[\text{TC}(\mathbb{S}_n^1) \leq 3.\]
\end{thm}

We provide an alternative proof to Tanaka's result that $\text{TC}(\mathbb{S}^1)=4$ that can be generalized to non-Hausdorff suspensions of finite spaces, and therefore finite models of higher dimensional spheres.

\begin{thm}\label{TCsusp}
If $X$ is a finite $T_0$ space and $Y:= X \varoast \mathbb{S}^0$ is the non-Hausdorff suspension of $X$, then 
\[
\text{TC}(Y) = \left\{ \begin{array}{ll} 1,&X \textrm{ is contractible}\\4,&X\textrm{ is not contractible}\end{array}\right.
\]
\end{thm}

As a consequence, $\text{TC}(\mathbb{S}^n)=4$ for $n \geq 1$, where $\mathbb{S}^n$ is the minimal finite model of $S^n$ comprising $2n+2$ points.

Lastly, we exhibit finite topological spaces weakly homotopy equivalent to a wedge of circles with arbitrarily high topological complexity. 
Properties of the Lusternik-Schnirelmann category can be used to show that $\text{TC}(\bigvee_n S^1) = 3$ where $\bigvee_n S^1$ is a wedge of $n$ copies of $S^1$.
Contrastingly, as $n$ increases, so does the topological complexity of a finite space weakly homotopy equivalent to $\bigvee_n S^1$.
We show this by proving the following result about non-Hausdorff joins of discrete spaces. 

\begin{thm}\label{TCjoin}
Let $X = \{x_1,\hdots,x_m\}$ and $Y=\{y_1,\hdots,y_n\}$ be finite spaces with $m,n>1$, each equipped with the discrete topology.
Take their non-Hausdorff join to be $Z:=X\varoast Y$.
Then $\text{TC}(Z) = n^2$ and $\text{TC}(Z^\text{op})=\text{TC}(Y\varoast X)=m^2$.
\end{thm}

\subsection{Notions of Topological Complexity}

In this section, I will review different notions of topological complexity for simplicial complexes and finite spaces and compare how they behave when applied to different models of circles.

It is important to note that Farber's original definition of topological complexity introduced in \cite{Farber2001} uses the unreduced Schwarz genus.
The best-known upper and lower bounds for topological complexity are
\[\textrm{zcl}(X) < \textrm{TC}(X) \leq \textrm{cat}(X \times X),\]
where the {\bf zero-divisors cup-length} $\text{zcl}(X)$ is the cup-length of $\ker(\Delta^*:H^*(X^2)\to H^*(X))$, and the {\bf Lusternik Schnirelmann category} $\text{cat}(X \times X)$ is the minimal number of open sets covering $X \times X$ whose inclusion map is nullhomotopic (these open sets are called {\bf categorical}).
Because of the strict inequality $\textrm{zcl}(X) < \textrm{TC}(X)$, many papers after Farber's subtract one from the definitions of $\text{TC}(X)$ and Lusternik-Schnirelmann category such that the upper- and lower-bounds may be equal in some cases.
All values of topological complexity given in this paper are unreduced, and we mention in the footnotes when this differs from an author's definition.

Although topological complexity has only been discussed over the last two decades, its formal definition draws from the Schwarz genus, defined in \cite{Schwarz1958} in 1958.
\begin{defn}
The {\bf Schwarz genus} $\mathfrak{g}(p)$ of a fibration $p:E \to B$ is the minimal number $k$ such that there exists an open covering $Q_1,\hdots,Q_k$ of $B$ with each $Q_i$ admitting a local $p$-section.
\end{defn}

\begin{defn}
Given a path-connected space $X$ and projection map $\pi:X^I \to X \times X$ that sends a path $\gamma$ to $\pi(\gamma)=(\gamma(0),\gamma(1))$, the {\bf topological complexity} of $X$, denoted $\textrm{TC}(X)$, is equal to $\mathfrak{g}(\pi)$.
\end{defn}

\begin{ex}
Farber proves $\text{TC}(S^1)=2$ by building an explicit motion planner.
The two open sets covering $S^1 \times S^1$ are given by
\begin{itemize}
\item $Q_1 = \{(x,y) \in S^1 \times S^1|x\neq-y\}$
\item $Q_2 = \{(x,y) \in S^1 \times S^1 | x\neq y\}$
\end{itemize}
with motion planner $s_1:Q_1 \to X^I$ traveling the shortest arc from $x$ to $y$ at constant speed, and $s_2:Q_2\to X^I$ moving at constant speed from $x$ to $y$ in a predetermined direction.
\end{ex}

In \cite{Gonzalez2017}, Gonz\'{a}lez defines an analog of topological complexity for simplicial complexes, called {\bf simplicial complexity} and denoted $\text{SC}(K)$ for a simplicial complex $K$.\footnote{In \cite{Gonzalez2017}, the definition of simplicial complexity is reduced.}
Gonz\'{a}lez' definition is adapted from Iwase and Sakai's intepretation of topological complexity as a fibrewise Lusternik-Schnirelmann category, introduced in \cite{Iwase2012}.
Their notion agrees with Farber's topological complexity of the geometric realization of $K$, as proven in Theorem 1.6 of \cite{Gonzalez2017}:
\[SC(K)=\textrm{TC}(|K|)\]

As a consequence, $\text{SC}(K)=2$ for any complex whose realization has the homotopy type of an odd sphere.
In particular, this includes $K$ such that $|K|\simeq S^1$.
They demonstrate this with $S^1$ modeled by the 1-skeleton of the 2-dimensional simplex $\Delta^2$, denoted $S_1$.
The open sets of $S_1 \times S_1$ admitting continuous motion planning rules follow Farber's construction closely. One collapses to $\{(x,x) \in S_1\times S_1\}$, and one to $\{(x,-x) \in S_1 \times S_1\}$.

While Gonz\'{a}lez' definition of topological complexity for simplicial complexes involves taking repeated barycentric subdivisions, the definition of {\bf discrete topological complexity} in \cite{Fernandez-Ternero2018} is defined in purely combinatorial terms.\footnote{In \cite{Fernandez-Ternero2018}, the definition of discrete topological complexity is reduced.}
Fern\'{a}ndez-Ternero, et al. prove in Example 4.9 of that paper that the minimal simplicial model of $S^1$, which is the boundary of a 2-simplex, has discrete topological complexity equal to 3.
For larger simplicial models of $S^1$, Theorem 5.6 of that paper proves the topological complexity drops back down to 2.

Tanaka introduces {\bf combinatorial complexity} (CC) in \cite{Tanaka2018} as an analog of topological complexity for finite spaces.
Tanaka's definition differs from Farber's in that they consider finite models of the interval in place of $I$.
Theorem 3.6 of \cite{Tanaka2018} proves the following:
\[\textrm{It holds that $\text{TC}(X) = \text{CC}(X)$ for any connected finite space $X$.}\]
Because connected finite spaces are path-connected by Proposition 1.2.4 of \cite{Barmak2012}, this is sufficient for defining a notion of topological complexity.
In Example 4.5 of that paper, Tanaka proves that $\text{TC}(\mathbb{S}^1)=4$, which is the result that motivated this paper.

\subsection{Finite Topology}

A thorough review of finite topology can be found in \cite{Barmak2012}.
Here, we define only what is necessary for the context of this paper.
Assume all finite spaces mentioned are $T_0$.

A finite topological space $X$ yields a preorder, $\leq$.
Given a point $x \in X$, its {\bf minimal open neighborhood}, or downset, is $x^\downarrow = \{y \in X \mid y \leq x\}$,
and its {\bf closure}, or upset, is $x^\uparrow = \{y \in X\mid y \geq x\}$.
If $x^\downarrow=x$, then $x$ is an open point; if $x^\uparrow = x$, then $x$ is a closed point.
We say these points are {\bf minimal} and {\bf maximal}, respectively.
For all $x \in X$, both $x^\downarrow$ and $x^\uparrow$ are contractible.
A point $x \in X$ is {\bf beat} if either $x^\uparrow-\{x\}$ has a unique minimal element, or $x^\downarrow-\{x\}$ has a unique maximal element.
Two points $x$ and $y$ are {\bf adjacent} if $x \in y^\downarrow$ or $y \in x^\downarrow$.

A finite $T_0$ space $X$ generates a simplicial complex, $\mathcal{K}(X)$, whose simplices are chains in $X$.
There exists a weak homotopy equivalence $\mu_X: |\mathcal{K}(X)| \to X$ called the {\bf $\mathcal{K}$-McCord map} that sends a point $\alpha \in |\mathcal{K}(X)| $ to $\min(\text{support}(\alpha)) \in X$.

\begin{ex}
A finite model of an interval can be thought of as a space $J_m = \{x_0,x_1,\hdots,x_m\}$ with its order given by $x_0 \leq x_1 \geq x_2 \leq \hdots \lessgtr x_m$.
Equivalently, the minimal open sets forming the basis for the topology on $J_m$ are given by $\{x_i\}$ if $i$ is even and $\{x_{i-1},x_i,x_{i+1}\}$ if $i$ is odd.
It is an easy exercise to show that $|\mathcal{K}(J_m)| \simeq I$.

Such spaces $J_m$ are the analogs of $I$ that \cite{Tanaka2018} uses in their definition of combinatorial complexity.
\end{ex}

\section{Larger finite models of $S^1$}

As a consequence of the examples given in the previous section, it is apparent that topological complexity is not invariant under weak homotopy type.
A reasonable hypothesis might be that $\textrm{TC}(\mathbb{S}^1)=4$ for all finite models of $\mathbb{S}^1$, but this is not the case.

\begin{defn}
The {\bf finite model of $S^1$} with $2n$ points for $n>2$ is the finite topological space \[\mathbb{S}^1_n = \{x_0,x_1,\hdots,x_{n-1},y_0,y_1,\hdots,y_{n-1}\}\] with the minimal open sets $x_i^\downarrow = \{x_i\}$ and $y_i^\downarrow =\{y_i,x_i,x_{i-1\bmod n}\}$ for $0 \leq i < n$.
\end{defn}

The minimal finite model of $S^1$ has two maximal points and two minimal points, which we denote as $\mathbb{S}^1_2$ or $\mathbb{S}^1$ when the context is unambiguous.
We start to prove Theorem \ref{TCS13} by limiting its upperbound.
It may be useful to refer to Figure \ref{S13}, which depicts $\mathbb{S}^1_3 \times \mathbb{S}^1_3$ with gray lines drawn between adjacent points as determined by the product topology.
Note that this visualization has been ``flattened''. 
The edges $\{x_0\}\times \mathbb{S}^1_3$ and $\{y_2\}\times \mathbb{S}^1_3$ are adjacent, and the edges $\mathbb{S}^1_3 \times \{x_0\}$ and $\mathbb{S}^1_3 \times \{y_2\}$ are adjacent.

\begin{thm}
\label{catS13}
Let $\mathbb{S}^1_n$ be the finite model of $S^1$ with $2n$ points, for $n \geq 3$.
Then $cat(\mathbb{S}_n^1\times\mathbb{S}_n^1)\leq 3$.
\end{thm}
\begin{proof}
Let $\mathbb{S}^1_n$ be as described above, with $n \geq 3$.
We can construct a covering by three contractible open sets, given by

\[Q_i := (\mathbb{S}^1_n-\{y_i\})\times(\mathbb{S}^1_n-\{y_i\}),\]

for $i=1,2,3$.
Since each $y_i$ is a closed point of $\mathbb{S}^1_n$, $\mathbb{S}_n^1- \{y_i\}$ is an open set.
To verify that $\mathbb{S}_n^1- \{y_i\}$ is contractible, notice that $x_{i-1}$ and $x_i$ are beat points of $\mathbb{S}_n^1- \{y_i\}$, so they can be removed while preserving homotopy type.
Next, $y_{i-1}$ and $y_{i+1}$ are beat points of $\mathbb{S}_n^1- \{y_i,x_i,x_{i-1}\}$ that can be removed.
This process can be repeated until only one point remains.
Hence $\mathbb{S}_n^1- \{y_i\}$ is contractible, and so each $Q_i$ is the product of an open contractible space with itself, which is again open and contractible.

It remains to be shown that these three sets cover all of $\mathbb{S}_n^1 \times \mathbb{S}_n^1$.
The first two sets $Q_1 \cup Q_2 =(\mathbb{S}^1_n \times \mathbb{S}^1_n) - \{(y_2,y_1),(y_1,y_2)\}$ cover all but two points.
Since neither of the uncovered points contain an instance of $x_3$, they are both included in $Q_3$.
\end{proof}

\begin{figure}
\begin{tikzpicture}

\draw[lightgray, very thin] (1,1) grid (6,6);

\foreach \x in {1,3,5}
	{
		\draw[lightgray, very thin] (\x,1)--(6,7-\x);
		\draw[lightgray, very thin] (\x+1,6)--(6,\x+1);
		\draw[lightgray, very thin] (1,\x)--(7-\x,6);
		\draw[lightgray, very thin] (\x,1)--(1,\x);
	}


\foreach \x in {1,3,5}
	\foreach \y in {1,3,5}
	{
		\draw (\x,\y) node [shape=circle,fill=black,scale=.5,draw]{};
	}
\foreach \x in {2,4,6}
	\foreach \y in {2,4,6}
	{
		\draw (\x,\y) node [shape=circle,fill=black,scale=.5,draw]{};
	}
\foreach \x in {2,4,6}
	\foreach \y in {1,3,5}
	{
		\draw (\x,\y) node [shape=circle,fill=black,scale=.5,draw]{};
		\draw (\y,\x) node [shape=circle,fill=black,scale=.5,draw]{};
	}

\draw (1,0.5) node {$x_0$};
\draw (2,0.5) node {$y_0$};
\draw (3,0.5) node {$x_1$};
\draw (4,0.5) node {$y_1$};
\draw (5,0.5) node {$x_2$};
\draw (6,0.5) node {$y_2$};

\draw (0.5,1) node {$x_0$};
\draw (0.5,2) node {$y_0$};
\draw (0.5,3) node {$x_1$};
\draw (0.5,4) node {$y_1$};
\draw (0.5,5) node {$x_2$};
\draw (0.5,6) node {$y_2$};

\end{tikzpicture}
\caption{A flattened top-down visualization of $\mathbb{S}^1_3\times \mathbb{S}^1_3$ in which the top and bottom edges adjacent, and the left and right edges adjacent.}
\label{S13}
\end{figure}
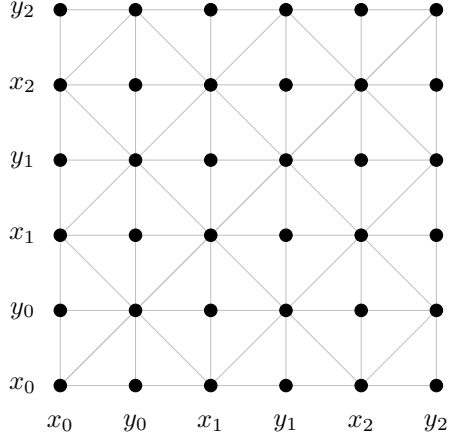

Theorem \ref{TCS13} follows immediately.
Note that the construction of each $Q_i$ above was somewhat arbitrary; many such constructions exist.
One might wonder if $\textrm{TC}(\mathbb{S}^1_n) = \textrm{TC}(S^1)=2$ when $n\geq3$.
In \cite{Gonzalez2017}, the author covers $S_1 \times S_1$ by two sets that collapse onto the diagonal and antidiagonal.
This is not possible for $\mathbb{S}^1_3$ because the elements in the antidiagonal of $\mathbb{S}_3^1 \times \mathbb{S}_3^1$ form a disconnected set, as can be seen in Figure \ref{S13}.
Because there is no connected set that can retract onto a disconnected set, we must pursue an alternate approach to determining $\text{TC}(\mathbb{S}_n^1)$ when $n \geq 3$.
We start by proving the following result, which holds in general for (not necessarily finite) path-connected spaces.

\begin{thm}\label{TCxX}
Let $X$ be a path-connected topological space and $Q \subseteq X \times X$ admitting a continuous section $s$ of the projection map $\pi:X^I \to X \times X$.
Then $X \times \{*\} \subseteq Q$ or $\{*\} \times X \subseteq Q$ if and only if $X$ is contractible.
\end{thm}

\begin{proof}
Suppose $Q \subseteq X \times X$ admits a continuous $\pi$-section $s$ and, without loss of generality, contains $X \times \{x_0\}$ as a subset  for some $x_0 \in X$.
Let $\sigma$ be the restriction of $s$ to $X \times \{x_0\}$, and consider the following diagram associated to the topological complexity of $X$.
Note that $\Delta = \pi \circ c$, and $\pi \simeq \Delta \circ d$.
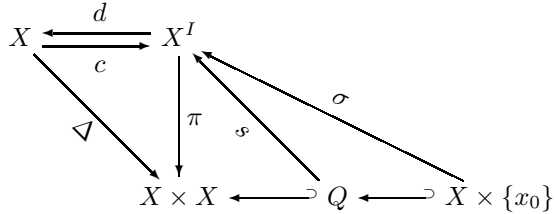
\begin{diagram}
X & \pile{\lTo^d \\ \rTo_{c}}&X^I\\
& \rdTo(2,2)_\Delta&\dTo_\pi&\luTo_s \luTo(4,2)^\sigma \\
&& X\times X &\lInto&Q&\lInto&X\times\{x_0\}
\end{diagram}

The map $c$ sends a point in $X$ to the constant path at that point in $X^I$, and $d$ is its homotopy inverse.
Given $\gamma \in X^I$, $d$ may be taken to be $d(\gamma) := \gamma(0)$.
Let $\iota_\ell: X \to  X\times X$ be inclusion into the left factor sending $x$ to $(x,x_0)$ (equivalently, $\iota_\ell:X\times\{x_0\}\to X \times X$ because $X$ and $X \times \{x_0\}$ are homeomorphic).
Then:
\begin{eqnarray*}
\iota_\ell &\simeq&\pi \circ \sigma\\
&\simeq&\pi \circ \text{id}_{X^I} \circ \sigma\\
&\simeq&\pi \circ [c \circ d] \circ \sigma\\
&\simeq&[\pi \circ c] \circ [d \circ \sigma]\\
&\simeq&\Delta \circ \text{id}_X\\
&\simeq&\Delta.
\end{eqnarray*}

Now, if $\iota_\ell \simeq \Delta$, consider their composition with $\text{pr}_2: X \times X \to X$ that projects onto the second factor.
Then $\text{pr}_2 \circ \iota_\ell \simeq \text{pr}_2 \circ \Delta$.
But then $\text{pr}_2 \circ \iota_\ell(x) = \text{pr}_2(x,x_0) = x_0$, and $\text{pr}_2 \circ \Delta(x) = \text{pr}_2(x,x) = x$.
It follows that $\text{id}_X$ is the constant map at $x_0$, which is only the case when $X$ is contractible.

To see the converse, suppose $X$ is contractible.
Then $X \times X$ is a contractible open set covering $X \times X$, hence it admits a continuous $\pi$-section.
\end{proof}

By Theorem \ref{TCxX}, if we are to cover $\mathbb{S}_n^1\times \mathbb{S}_n^1$ with only two open sets that each admit a continuous section, each set can contain neither $\{z\}\times \mathbb{S}_n^1$ nor $\mathbb{S}_n^1\times \{z\}$ as a subset for all $z \in \mathbb{S}_n^1$.
Because $\mathbb{S}_3^1\times \mathbb{S}_3^1$ contains nine maximal elements, this problem is analogous to shading five squares of a $3\times3$ grid such that no column or row is shaded, and such that no vertical or horizontal line between colums and rows is shaded.
We invite the reader to examine Figure \ref{S13} to see that this is not possible.
Because the open sets constructed in Theorem \ref{catS13} are contractible, each $Q_i$ admits a continuous motion planner.
\begin{cor}
$\text{TC}(\mathbb{S}_3^1)=3$.
\end{cor}

When $n \geq 5$, we can exhibit a covering of $\mathbb{S}^1_n \times \mathbb{S}^1_n$ by two open sets that avoid containing $\mathbb{S}^1_n \times \{*\}$ or $\{*\} \times \mathbb{S}^1_n$:

\begin{ex}
Consider $\mathbb{S}^1_n$ with $n \geq 5$.
Take $D:=\{(x,-x)\subset \mathbb{S}^1_n \times \mathbb{S}^1_n\}^\uparrow$ to be the closure of the anti-diagonal.
Define $Q_1 := (\mathbb{S}^1_n\times \mathbb{S}^1_n)- D$ and $Q_2 := D^\downarrow$.
Because this is a covering of $\mathbb{S}^1_n \times \mathbb{S}^1_n$ by two open sets satisfying the hypotheses of Theorem \ref{TCxX}, and because $\text{TC}(S^1)=2$, we have reason to believe that $\text{TC}(\mathbb{S}^1_n)=2$ when $n\geq 5$, however, we know of no explicit motion planner on these sets.
\end{ex}

\section{Suspensions and Wedges}

\subsection{Suspensions}

It is well known that $\text{TC}(S^n)=2$ for $n$ odd and $3$ for $n$ even.
Here, we show the minimal finite models of $S^n$ built of iterated non-Hausdorff suspensions of $S^0$ have $\text{TC}(\mathbb{S}^n)=4$ for all $n$.

\begin{defn}
The {\bf non-Hausdorff join} of two finite $T_0$ spaces $X$ and $Y$ is given by $X \varoast Y := X \sqcup Y$ with each of $X$ and $Y$ keeping its given ordering, along with $x \leq y$ for all $x \in X, y \in Y$.
If $Y=S^0 = \mathbb{S}^0$, then $X \varoast Y$ is the {\bf non-Hausdorff suspension}.
\end{defn}

Let $\mathbb{S}^0 = \{x_0,y_0\}$ equipped with the discrete topology be the minimal finite model of $S^0$.
We can iteratively construct minimal finite models of spheres by taking the non-Hausdorff suspension of each $\mathbb{S}^n$.
That is, $\mathbb{S}^n = \mathbb{S}^{n-1} \varoast \mathbb{S}^0$.
See that the minimal finite model of any $n$-sphere has two maximal elements.
By Corollary 3.6 of \cite{Tanaka2018}, this means $\text{TC}(\mathbb{S}^n)\leq 4$.
We will now provide an alternative proof of Tanaka's Example 4.5 that can be generalized to the minimal finite model of any $n$-sphere for $n \geq 1$.

\begin{proof}[Proof of Theorem \ref{TCsusp}.]
Let $Y = X \varoast \mathbb{S}^0$ be the non-Hausdorff suspension of any finite $T_0$ space $X$.
The open sets of $Y$ are the open sets of $X$, together with $\{x_0 \cup X\}$ and $\{y_0 \cup X\}$;
the two maximal elements of $Y$ are $x_0$ and $y_0$.
Then $Y\times Y$ has four maximal elements: $\{(x_0,x_0),(x_0,y_0),(y_0,x_0),(y_0,y_0)\}$.
From this, we get $\text{TC}(Y)\leq\text{cat}(Y\times Y)\leq4$.

When $X$ is contractible, $Y=X\varoast S^0$ is contractible by Proposition 2.7.3 of \cite{Barmak2012}, so $\text{TC}(Y)=1$.

Assume $X$ (and therefore $Y$) is not contractible.
If it were to be the case that $\text{TC}(Y)<4$, then one of the open sets $Q \subseteq Y\times Y$ must contain two of those four maximal elements.
Call these distinct maximal elements $(m_1,m_2),(m_1',m_2') \in \{(x_0,x_0),(x_0,y_0),(y_0,x_0),(y_0,y_0)\}$.
Notice that $(m_1,m_2)^\downarrow = m_1^\downarrow \times m_2^\downarrow$.

If $m_1 = m_1'$, then $m_2 \neq m_2'$ because the maximal elements are distinct, so 
\begin{eqnarray*}
Q &\supseteq& (m_1,m_2)^\downarrow \cup (m_1,m_2')^\downarrow\\
&=&(m_1^\downarrow \times m_2^\downarrow) \cup (m_1^\downarrow \times m_2'^\downarrow)\\
&=& m_1^\downarrow \times (m_2^\downarrow \cup m_2'^\downarrow)\\
&=&m_1^\downarrow \times Y\\
&\supseteq& \{m_1\} \times Y,
\end{eqnarray*}
contradicting the assumption that $Y$ is not contractible. Similarly, if $m_2=m_2'$, then $Y \times \{m_2\} \subseteq Q$.

If $m_1 \neq m_1'$ and $m_2 \neq m_2'$, notice for example that there exists an $x \in m_1^\downarrow \cap m_1'^\downarrow = X \subset Y$.
Then
\begin{eqnarray*}
Q &\supseteq &(m_1,m_2)^\downarrow \cup (m_1',m_2')^\downarrow\\
&=&(m_1^\downarrow \times m_2^\downarrow) \cup (m_1'^\downarrow \times m_2'^\downarrow)\\
&\supseteq& (x\times m_2^\downarrow) \cup (x \times m_2'^\downarrow)\\
&=&\{x\} \times (m_2^\downarrow \cup m_2'^\downarrow)\\
&=&\{x\} \times Y,
\end{eqnarray*}
again a contradiction.

%

Hence any covering of $Y\times Y$ by fewer than four open sets cannot admit a continuous motion planner.
Hence $\text{TC}(Y)=4$ when $X$ is not contractible.
\end{proof}

\begin{cor}
For $n>0$, $\text{TC}(\mathbb{S}^n) = 4$.
\end{cor}

\begin{proof}
Taking $X = \mathbb{S}^{n-1}$, this follows from Theorem \ref{TCsusp}.
\end{proof}

The technique of Example 4.5 in \cite{Tanaka2018} works, in fact, for the join of any two discrete spaces.

\begin{proof}[Proof of Theorem \ref{TCjoin}.]
By the non-Hausdorff join, the minimal open sets of $Z$ are $x_i^\downarrow =\{x_i\}$ and $y_j^\downarrow=\{y_j\}\cup X$ for $i \in \{1,\hdots,m\}$ and $j \in \{1,\hdots,n\}$.
Because downsets are contractible, $Z$ can be covered by $n^2$ contractible open sets of the form $(y_i,y_j)^\downarrow = y_i^\downarrow \times y_j^\downarrow$ for $i,j \in \{1,\hdots,n\}$.
This gives an upperbound for $\text{cat}(Z\times Z)$, hence an upperbound for $\text{TC}(Z)\leq n^2$.
If $\text{TC}(Z)<n^2$, an open set admitting a continuous section $s:Q\to Z^I$ must contain at least two maximal elements of $Z\times Z$.
Note that since $X \subset y^\downarrow$ for all maximal $y \in Z$, $X \times X \subset (y_1,y_2)^\downarrow$ for all maximal $(y_1,y_2) \in Z \times Z$.
We can apply Tanaka's argument in Example 4.5 of \cite{Tanaka2018} that proves $\text{TC}(\mathbb{S}^1)=4$:

Since $X\times X \subseteq Q$, there exists an $( x_i , x_j ) \in Q$ with $ x_i \neq x_j$.
Because $\{ x_i , x_j \}$ is a disconnected space, $s( x_i , x_j )$ must pass through some point of $Y \subset Z$.
Note $(x_j,x_i) \in Q$ as well.
Let \[u:=\min\{s( x_i , x_j )^{-1}(Y),s( x_j , x_i )^{-1}(Y)\}\] and \[v:=\max\{s( x_i , x_j )^{-1}(Y),s( x_j , x_i )^{-1}(Y)\}.\]

This means
\[s(x_i,x_j)(t) = \left\{\begin{array}{ll}x_i, & t \in [0,u)\\ x_j,&t \in (v,1] \end{array}\right.
\textrm{ and }
s(x_j,x_i)(t) = \left\{\begin{array}{ll}x_j, & t \in [0,u)\\ x_i,&t \in (v,1] \end{array}\right..\]

Let $(m,m')\in\{(y_i,y_j),(y_i',y_j')\}$ be an arbitrary maximal element of $Q$.
Since $( x_i , x_j )\leq(m,m')$, $s( x_i , x_j )\leq s(m,m')$ by the continuity of $s$.
Similarly, $s( x_j , x_i )\leq s(m,m')$.
Since $s(m,m')|_{[0,u)} \geq  x_i$ and $s(m,m')|_{[0,u)} \geq  x_j $, $s(m,m')$ can never be minimal on that interval, and so $s(m,m')|_{[0,u)}=m$.
Similarly, $s(m,m')|_{(v,1]}=m'$.

By the construction of $u$, at least one of either $s( x_i , x_j )(u)\in Y$ or $s( x_j , x_i )(u)\in Y$.
Without loss of generality, suppose $s( x_i , x_j )(u)=y \in Y$.
Then $s(m,m')(u)\geq s( x_i , x_j )(u)=y$ implies $s(m,m')(u)=y$.
Since $s(m,m')|_{[0,u)}$ is never minimal, it must follow that $s(m,m')|_{[0,u)}=y$ as well, forcing $m=y$.
The choice of $(m,m')$ was arbitrary, so $y_i=y=y_i'$.
By a similar argument, $y_j=y_j'$.

Hence $(y_i,y_j)=(y_i',y_j')$, contradicting our assumption that $Q$ contained two distinct maximal elements, so $\text{TC}(Z)=n^2$.

Because $Z^\text{op}= Y \varoast X$, a similar argument works to show $\text{TC}(Z^\text{op})=m^2$.
\end{proof}

\subsection{Wedges}

In \cite{Zapata2017}, the author proves that \[\text{TC}(X \vee Y) = \max\{\text{TC}(X),\text{TC}(Y),\text{cat}(X \times Y)\}.\]
As a consequence, $\text{TC}(S^1 \vee S^1)=3$.
Any wedge of circles $\bigvee S^1 = \bigsqcup [0,1]/(0\sim1)$ can be covered in two categorical open sets: 
$Q_1 := \bigsqcup (-\epsilon, \epsilon)/0$ for some $0<\epsilon<1$ and 
$Q_2 := \bigsqcup (0,1)$.
By Proposition 2.3 of \cite{James1978}, $\text{cat}(\bigvee S^1 \times \bigvee S^1) \leq 3$, so $\text{TC}(\bigvee S^1) \leq 3$.
If we were to have $\text{TC}(\bigvee S^1) =2$, then $\bigvee S^1$ would have the homotopy type of an odd sphere\footnote{In \cite{Grant2012}, they use the reduced definition of topological complexity.} by Theorem 1 of \cite{Grant2012}. 
Hence, $\text{TC}(\bigvee S^1)=3$.

James proves in Proposition 2.3 of \cite{James1978} that $\text{cat}(X \times X) < 2\text{cat}(X)$.
As stated in Remark 2.7 of \cite{Tanaka2018}, this result does not hold in general for finite spaces.
Here, we prove a result weaker than James', and stronger than the upperbound $\text{cat}(X \times X) \leq (\text{Max}(X)^\#)^2$ proven in Corollary 3.8 of \cite{Tanaka2018}.

\begin{prop}
\label{catX2}
Given a finite space $X$, $\text{cat}(X \times X) \leq \text{cat}(X)^2$.
\end{prop}
\begin{proof}
Let $X$ be a finite space and $\{Q_i\}_{i=1}^k$ be a categorial covering by $k$ open sets.
Then $X \times X$ can be covered in $k^2$ open sets $\{Q_i \times Q_j\}_{1 \leq i,j \leq k}$.
Each $Q_i$ has an associated homotopy $h_i: Q_i \times I \to X$ such that $h_i(x,0) = x$ and $h_i(x,1) = x_i$ for some constant $x_i \in X$.
For any $(i,j)$ pair, the product $(h_i \times h_j): Q_i \times I \times Q_j \times I \to X \times X$ is also continuous.
Consider the map $H_{i,j}:Q_i \times Q_j \times I \to X \times X$ given by $H_{i,j}(x,x',t) := (h_i \times h_j)(x,t,x',t)$.
At $t=0$, $H_{i,j}(x,x',0) = (h_i \times h_j)(x,0,x',0) = (h_i(x,0),h_j(x',0)) = (x,x')$ is the inclusion map, and at $t=1$,
$H_{i,j}(x,x',1) = (h_i \times h_j)(x,1,x',1) = (h_i(x,1),h_j(x',1)) = (x_i,x_j)$ is the constant map at the point $(x_i,x_j)$.
Hence each inclusion map $Q_i \times Q_j \hookrightarrow X \times X$ is nullhomotopic, so $\text{cat}(X \times X) \leq \text{cat}(X)^2$.
\end{proof}

\begin{prop}
Let $\bigvee_{i=1}^m \mathbb{S}^1_{n_i}$ with $n_i \geq 2$ be a wedge of $m$ finite models of $\mathbb{S}^1$, each with $2n_i$ points.
Then $\text{TC}(\bigvee_{i=1}^m \mathbb{S}^1_{n_i})\leq 4$.
\end{prop}

\begin{proof}
Let $\bigvee_{i=1}^m \mathbb{S}^1_{n_i}$ be as defined above with each copy of $\mathbb{S}^1_{n_i}$ identified at $y_0:=y_{0_1}=y_{0_2}=\hdots=y_{0_m}$.
We can cover $\bigvee_{i=1}^m \mathbb{S}^1_{n_i}$ in two categorical open sets.
Take $Q_1:= y_0^\downarrow$.
Because downsets are contractible, $\iota: Q_1 \hookrightarrow \bigvee_{i=1}^m \mathbb{S}^1_{n_i}$ is nullhomotopic.
Take $Q_2:= \bigsqcup_{i=1}^m \mathbb{S}^1_{n_i}-\{y_{0_i}\}$.
Each of the $\mathbb{S}^1_{n_i}-\{y_{0_i}\}$ are disjoint and contractible, and $\bigvee_{i=1}^m \mathbb{S}^1_{n_i}$ is connected, so the inclusion map $\iota:Q_2 \hookrightarrow \bigvee_{i=1}^m \mathbb{S}^1_{n_i}$ is nullhomotopic.
Applying Proposition \ref{catX2} gives $\text{TC}(\bigvee_{i=1}^m \mathbb{S}^1_{n_i}) \leq \text{cat}(\bigvee_{i=1}^m \mathbb{S}^1_{n_i})^2\leq 4$.
\end{proof}

If $K$ is an abstract simplicial complex whose realization is a finite wedge of circles, with each circle triangulated by more than three edges, \cite{Fernandez-Ternero2018} shows that $\text{TC}(K)=3$.
It is unknown at this time if we can improve the bound $\text{TC}(\bigvee_{i=1}^m \mathbb{S}^1_{n_i}) \leq 4$.
We can exhibit a series of finite spaces that are also weakly homotopy equivalent to a wedge of circles, but whose topological complexity is arbitrarily high.

\begin{prop}
If $X$ and $Y$ are discrete finite spaces with $|X|=m$ and $|Y|=n$, then $X \varoast Y$ is weakly homotopy equivalent to a wedge of $(m-1)(n-1)$ circles.
\end{prop}

\begin{proof}
Let $X$ and $Y$ be defined as above and take $Z:=X \varoast Y$.
By definition, the McCord map $\mu: |\mathcal{K}(Z)| \to Z$ is a weak homotopy equivalence.
The geometric realization $|\mathcal{K}(Z)|$ has $m+n$ vertices and $mn$ 1-simplices.
Quotienting by any spanning tree yields a simplicial complex with one $0$-simplex and $mn-(m+n-1) = (m-1)(n-1)$ $1$-simplices, which is homotopy equivalent to a wedge of $(m-1)(n-1)$ circles.
Then \[\bigvee_{(m-1)(n-1)}S^1 \rightarrow |\mathcal{K}(Z)| \rightarrow Z\] is a weak homotopy equivalence.
\end{proof}

Below are some bizarre consequences of this and Theorem \ref{TCjoin}.

\begin{ex}
Let $\bigvee S^1$ be a wedge of $(m-1)(n-1)$ circles, and let $X$ and $Y$ be discrete spaces with $|X|=m$ and $|Y|=n$.
Then there exist McCord maps $\mu_1:\bigvee S^1 \to X \varoast Y$ and $\mu_2: \bigvee S^1 \to Y \varoast X$ such that $\text{TC}(\mu_1(\bigvee S^1))=n^2$ and $\text{TC}(\mu_2(\bigvee S^1))=m^2$.
\end{ex}

\begin{ex}
Let $Y$ be a discrete space with $n$ points.
For all discrete, finite $X$ such that $|X|\geq 2$, $\text{TC}(X \varoast Y) = n^2$.
\end{ex}

\section{Concluding Remarks}

It is of interest to note that all of the explicit computations of topological complexity for finite spaces rely on using the Lusternik-Schnirelmann category as an upper-bound.
Specifically, all currently known motion planners for finite spaces are defined on categorical open sets.
We are very interested in examples of finite spaces $X$ for which $\text{TC}(X) < \text{cat}(X \times X)$.

I would like to thank Kohei Tanaka for some insightful email correspondence about Example 4.5 of \cite{Tanaka2018}.


\end{document}